\documentclass[11pt,leqno,parskip=full,a4paper]{amsart}
\usepackage{txfonts}
\usepackage[T1]{fontenc}
\usepackage[latin1]{inputenc}
\usepackage{lmodern}
\usepackage{mathrsfs}
\usepackage[english]{babel}
\usepackage{amssymb,hyperref,amsmath,enumitem}
\theoremstyle{plain}
\newtheorem{theo}{Theorem}[section]

\newtheorem{lemm}[theo]{Lemma}

\theoremstyle{definition}
\newtheorem{defi}[theo]{Definition}

\newtheorem{assum}{Main Assumption\!\!}

\DeclareMathAlphabet{\mathrmsl}{OT1}{cmr}{m}{sl}
\renewcommand{\leq}{\leqslant}
\renewcommand{\geq}{\geqslant}

\newcommand{\TM}{{\mathbb T}}
\newcommand{\RM}{{\mathbb R}}
\newcommand{\ZM}{{\mathbb Z}}
\newcommand{\mV}{{\mathbb V}}
\newcommand{\mW}{{\mathbb W}}

\newcommand{\vol}{\operatorname{vol}}
\newcommand{\Scal}{\operatorname{Scal}}
\newcommand{\tr}{\operatorname{tr}}
\newcommand{\ptr}{\operatorname{p-tr}}

\newcommand{\Spon}{\operatorname{SO}}
\newcommand{\Spin}{\operatorname{Spin}}

\renewcommand{\Im}{\operatorname{Im}}

\newcommand{\diver}{\operatorname{div}}

\newcommand{\son}{\mathfrak{so}(n)}

\newcommand{\proofof}[1]{\end{#1}\begin{proof}}

\newcounter{mnotecount}[section]
\renewcommand{\themnotecount}{\thesection.\arabic{mnotecount}}
\newcommand{\mnote}[1]
{\protect{\stepcounter{mnotecount}}$^{\mbox{\footnotesize  $
      \bullet$\themnotecount}}$ \marginpar{\raggedright\tiny\em
    $\!\!\!\!\!\!\,\bullet$\themnotecount: #1} }

\swapnumbers

\begin{document}
\title{Universal positive mass theorems}
\author{Marc Herzlich}
\thanks{The author is supported in part by the project ANR-12-BS01-004 `Geometry and topology of open manifolds' of the French National Agency for Research.}
\address{Universit\'e de Montpellier\\ Institut Montpelliérain Alexander Grothen\-dieck\\ UMR 5149 CNRS -- UM\\ Montpellier\\ France} 
\email{marc.herzlich@umontpellier.fr}
\date{\today}
\begin{abstract} 
In this paper, we develop a general study of contributions at infinity of Bochner-Weitzenb\"ock-type formulas on asymptotically flat manifolds, inspired by Witten's proof of the positive mass theorem. As an application, we show that similar proofs can be obtained in a much more general setting as any choice of an irreducible natural bundle and a very large choice of first-order operators may lead to a positive mass theorem along the same lines if the necessary curvature conditions are satisfied.
\end{abstract}

\maketitle

\section*{Introduction}

Mass is the most fundamental invariant of asymptotically flat manifolds. Originally defined by physicists in
General Relativity, it has played a leading role in conformal geometry 
\cite{lp} or $3$-dimensional Riemannian geometry \cite{akutagawa-neves,bray-neves}. The most important 
feature of mass is its positivity in the presence of nonnegative scalar curvature and the subsequent rigidity
statement (zero mass implies flatness).
The positive mass theorem was proved first by Schoen and Yau in dimensions between $3$ and $7$ using minimal surfaces \cite{SY1}, and Witten introduced in 1981 a new method, based on spinors and a Bochner-type formula for the Dirac operator usually known as the Lichnerowicz-Schr\"odinger formula \cite{witten}.
Its efficiency made it usable in a variety of other contexts: for instance, a lot of generalizations of mass have been introduced 
in the recent years, and almost all statements of their positivity were proved along Witten's lines, 
see \cite{mh-ptc, dai, maerten, maerten-minerbe, minerbe-alf-mass}.

The most striking feature of Witten's proof is the appearance of the mass as the boundary-at-infinity contribution in the Lichnerowicz-Schr\"odinger formula. The goal of this paper is to try to analyze further this idea, by devising a general way to compute the boundary-at-infinity contribution for a
very general class of Bochner-type formulas. We shall then show that Witten's argument may be applied to a very large class of natural bundles and operators, similar to the Dirac operator and giving birth to a Bochner-type formula, and that there is \emph{always} a connection between the boundary-at-infinity contribution and the mass. This may lead to positive (or negative, sometimes) mass theorems in some special cases. The main motivation of the present paper also lies in the methods. Our hope is that the analysis performed here for the mass may find useful applications, either in problems where similar boundary-at-infinity contributions play a role, or in the study of newly defined asymptotic invariants for which positivity statements have not been proven in full generality so far.

\subsection*{Acknowledgements} The author is grateful to Julien Cortier and Mattias Dahl for numerous discussions on the subject of this paper. He is also grateful to the referee for his very careful reading of the paper, which led to many improvements.
 
\section{Asymptotically flat manifolds, mass, and the Witten argument}\label{sec:witten}

In all that follows, the dimension $n$ 
of the manifolds considered will be taken to be at least $3$. We shall restrict ourselves to manifolds with only one end, but the definitions can straightforwardly be extended to the general case.

\begin{defi}\label{def_aflat}
An asymptotically flat manifold is a complete Riemannian manifold $(M,g)$ of dimension $n\geq 3$ such that 
there exists a diffeomorphism $\Phi$ from the complement of a compact set $K$ in $M$ into the complement of a 
ball $B$ in $\RM^n$ (called a \emph{chart at infinity}), such that, in these coordinates,
$$ |g_{ij} - \delta_{ij} | = O(r^{-\tau}), \quad |\partial_kg_{ij}| = O(r^{-\tau-1}), 
\quad |\partial_k\partial_{\ell}g_{ij}| = O(r^{-\tau-2}) $$
for some real $\tau>0$, called the \emph{order} of asymptotic flatness.
\end{defi}

Asymptotically flat manifolds are natural objects in General Relativity, where they are used as models for isolated spacetimes with 
gravitational interactions vanishing at infinity. Mass has then been introduced in the physics literature as a way to measure the
total amount of matter contained in the spacetime.

\smallskip

\begin{defi}
Let $(M,g)$ be an asymptotically flat manifold of dimension $n\geq 3$ and order $\tau>0$.
If $\tau>\tfrac{n-2}{2}$ and the scalar curvature of $g$ is integrable, then the quantity
\begin{equation}\label{defn_mass} 
m(g) = \lim_{r\to\infty} \int_{S_r} (\diver_0 g - d \tr_0 g)(\nu)\,d\!\vol_{s_r} 
\end{equation}
(where $\nu$ denotes the field of outer unit normals to the coordinate spheres $S_r$ and the subscript $\cdot_0$ 
refers to the euclidean metric in a given chart at infinity) exists and is 
independent of the chosen chart \cite{ba}. It is called the \emph{mass} of the
asymptotically flat manifold $(M,g)$.
\end{defi}

\smallskip

The positive mass theorem states that if the scalar curvature of $g$ is nonnegative, then its mass
is nonnegative, and it vanishes if and only if $(M,g)$ is isometric to the Euclidean space. 
Witten's approach for the proofs of these two statements can then be described as follows: given any \emph{spin} asymptotically flat 
manifold $(M,g)$ and $\phi_0$ a nonzero constant spinor on $\RM^n$, one may find (by some PDE analysis in Witten's case, but this may not be an obligation) a spinor field
$\phi$ on $(M,g)$ satisfying
\begin{equation}\label{eqn_Dirac} \mathfrak{D}\phi = 0, \quad \phi \rightarrow_{\infty} \phi_0 \end{equation}
where $\mathfrak{D}$ is the Dirac operator \cite{ba,parker-taubes,witten}. The
Lichnerowicz-Schr\"odinger formula then relates the Dirac laplacian $\mathfrak{D}^*\mathfrak{D}$ to the 
rough laplacian $\nabla^*\nabla$ and the scalar curvature:
\begin{equation}\label{lichne}  \nabla^*\nabla + \frac14\Scal - \mathfrak{D}^*\mathfrak{D} = 0. \end{equation}
After integrating over domains bounded by coordinates spheres $S_r$, performing an integration by parts, and letting
$r\to\infty$, one gets
\begin{equation}\label{lichne_int} \int_M |\nabla\phi|^2 + \frac14\Scal|\phi|^2 - |\mathfrak{D}\phi|^2 = 
\lim_{r\to\infty} \int_{S_r} b(\phi) \end{equation}
where $b$ stands here for the integrand of the boundary term. Some analysis of the equations (\ref{eqn_Dirac}) above yields that all
integrals converge, and moreover (this is the core of Witten's approach) the boundary contribution is
given at infinity by 
$$ \lim_{r\to\infty} \int_{S_r} b(\phi) \ = \ \frac{1}{4}\lim_{r\to\infty}\int_{S_r} (\diver_0 g - d \tr_0 g)(\nu)|\phi_0|^2\,d\!\vol_{s_r}
\ = \ \frac{1}{4}|\phi_0|^2 m(g) .$$ 
As the left hand side in (\ref{lichne_int}) is nonnegative since $\phi$ is a solution of the Dirac equation, this proves the first part of the positive mass theorem. The rigidity part then follows from the fact that zero mass implies the existence of a parallel spinor: the metric $g$ is then Ricci-flat, and this forces $(M,g)$ to be isometric to the Euclidean space as a simple consequence of the Bishop-Gromov comparison theorem.

Witten's remarkable idea has been used again in numerous cases. For instance, Bartnik  showed that the Dirac operator on spinors could be replaced by the Hodge-de Rham operator $d+\delta$ on $1$-forms \cite{ba}. More recently, Jammes gave a proof of the positive mass theorem in dimension $4$ using the Hodge-de Rham operator on $2$-forms \cite{jammes}. (Interestingly enough, Jammes does not need analysis to solve the equation similar to (\ref{eqn_Dirac}) that appears in his proof). Spinors were also used as the main ingredient to prove positivity of generalizations of mass with other behaviours at infinity:  asymptotically hyperbolic manifolds \cite{mh-ptc, maerten}, manifolds asymptotic to a product of a torus by some Euclidean space \cite{dai}, to a fiber bundle over some Euclidean space \cite{minerbe-alf-mass}, or K\"ahler asymptotically complex hyperbolic manifolds~\cite{maerten-minerbe}. 

In all these works, the mass (or some similar invariant when other behaviours at infinity are involved) appears in the boundary-at-infinity contribution. This should not come as a surprise: in this paper, we shall show that the mass always appears as a part of the boundary-at-infinity contribution, independently of the chosen bundle and operator, provided that some Bochner-type formula is available, see Theorem \ref{theo:main} below. Moreover, it turns out that an explicit formula can always be given in terms of representation-theoretic data of the bundle and operator involved, without too specific or difficult computations. 

The paper is organized as follows. Section \ref{sec:setting} collects all the facts needed for our analysis. Our main result is stated in Section \ref{sec:main}, and some preliminary steps towards its proof are also made there. Section \ref{sec:computation} is devoted to the technical part of the proof. Applications are then proposed in Section~\ref{sec:appli}.

\section{General setting}\label{sec:setting}

Let $(M,g)$ be a Riemannian manifold of dimension $n\geq 3$, decay order $\tau>\tfrac{n-2}{2}$ and integrable scalar curvature, so that the mass is
defined. We now consider a natural bundle $E$ on $M$ issued from an \emph{irreducible} representation 
$(\mathbb{V}, \rho)$ of the special orthogonal group $\Spon(n)$ or its universal covering $\Spin(n)$. Thus 
$E=\mathcal{P}\times_\rho\mathbb{V}$ where $\mathcal{P}$ is either the frame bundle or the spin frame bundle of $(M,g)$, 
the manifold $M$ being of course assumed to be spin in the latter case.

To give an explicit description of our results, we shall need a few classical facts of representation theory of the Lie algebra
$\son$. From now on, we shall freely identify elements of $\RM^n$ and $(\RM^n)^*$, \emph{i.e.} forms and vectors ; 
this simplification will be used repeatedly and without notice in the sequel of the paper. Thus, for any pair
of vectors $u_1$ and $u_2$, we denote by $u_1\wedge u_2$ the element of $\son$ given by $X \mapsto \langle u_1, X\rangle u_2 - 
\langle u_2,X\rangle u_1$.

Let $\{e_j\}_{1\leq j\leq n}$ be an orthonormal basis of $\RM^n$ and $\mathfrak{h}$ be the Cartan subalgebra of $\son$ generated by 
$\epsilon_k = e_{2k-1}\wedge e_{2k}$ where $k$ runs from $1$ 
to $m=\lfloor\tfrac{n}{2}\rfloor$. Any complex representation $\mV$ of $\son$ may be split into eigenspaces for the action of 
$i\mathfrak{h}$ (whose elements
are all simultaneously diagonalizable), which are labelled by elements of $(i\mathfrak{h})^*$ called weights. All the weights
appearing for a given representation may be written in coordinates relative to the basis $\{\mu_j\}_{1\leq j\leq m}$ 
of $(i\mathfrak{h})^*$ defined by
$\mu_j(i\epsilon_k) = \delta_{jk}$, and the largest weight for the lexicographic order is called the \emph{dominant weight} of the 
representation. The main classification result of the theory then states that irrreducible complex representations $\mV$ are in one-to-one correspondence with 
their dominant weights. The set of admissible dominant weights (\emph{i.e.} those corresponding to some irreducible representation of $\son$) is the set of $m$-uplets $(\rho^1, \cdots,\rho^m)$ in $\left(\tfrac12\ZM\right)^m$ such that 
\begin{equation}
\label{eq:dominant}
\begin{cases}
\rho^1 \geq \cdots \geq \rho^m \geq 0 \ & \ \textrm{if } n=2m+1, \\ 
\rho^1 \geq \cdots \geq |\rho^m| 
\ & \ \textrm{if } n=2m. 
\end{cases}
\end{equation}
Whenever there exists a real representation space $\mV_{\RM}$ of $\son$ such that $\mV = \mV_{\RM}\otimes\mathbb{C}$, the representation $(\mV,\rho)$ is called \emph{real}. In this case, we shall always use below the real vector space $\mV_{\RM}$ as $\mV$, as staying in the realm of real numbers does not introduce any new
difficulty (we shall thus use the same notation $\mV$ in the real and complex cases).

Letting from now on $(\mV,\rho)$ be an \emph{irreducible} representation of $\son$, the tensor product $\RM^n\otimes \mV$ (tensored over the real numbers if $\mV$ is real, and over the complex numbers if not) splits under the action of $\son$ into $N$ irreducible components:
$$\RM^n\otimes \mV = \otimes_{j=1}^N \mW_j ,$$
where the representation of $\son$ on $\mW_j$ will be denoted by $\lambda_j$.
To know which irreducible representations appear in the tensor product, one may use the following rule: a weight $\lambda$ 
appears as the dominant weight of a summand $\mW$ in $\RM^n\otimes\mV$ iff both the following conditions hold together~\cite{fulton-harris}:
\begin{itemize}
\item[(1)] $\lambda = \rho \pm \mu_i$ for some $i$, or $\lambda = \rho$ when $n=2m+1$ and $\rho^m>0$,
\item[(2)] $\lambda$ is dominant, \emph{i.e.} it satisfies the conditions
given by (\ref{eq:dominant}).
\end{itemize}
At the bundle level, one obtains the corresponding splitting 
$$T^*M\otimes E =\otimes_{j=1}^N F_j.$$
Letting $\Pi_j : \RM^n\otimes \mV \rightarrow \mW_j$ be the projection onto the $j$-th summand (or the analogous map at the
bundle level), we shall follow \cite{homma-casimir} and denote by $p_{j}$ the \emph{generalized Clifford action} 
$p_{j} (X)\sigma = \Pi_j(X\otimes\sigma)$.
Each projection induces a natural first-order operator 
$P_j = \Pi_j \circ \nabla$ (where $\nabla$ is the Levi-Civita connection) sending sections of $E$ into sections
of $F_j$, known as a \emph{Stein-Weiss operator}. The principal symbol of $P_j$ is $p_j(\xi)$ and that of
$(P_j)^*P_j$ is $p_j(\xi)^*p_j(\xi)$.

We shall also need the so-called \emph{conformal weight operator}, which plays an important role in conformal geometry. 
It is usually described as the operator $B : \RM^n\otimes\mV \rightarrow \RM^n\otimes\mV $ defined by
$$ B(\alpha\otimes v) = \sum_{i=1}^n e_i \otimes \rho (e_i\wedge\alpha)v ,$$
and, as such, is equivariant under the action of $\son$. The operator $B$ is self-adjoint, its eigenvalues are then real, and, by Schur's Lemma, $B$ acts homothetically 
on each summand
appearing in the decomposition of $\RM^n\otimes\mV$ into irreducibles. 
For any such summand $(\mW,\lambda)$, the corresponding eigenvalue is
$$ w(\lambda,\rho) \ = \ \frac12\,\left( c(\lambda) - c(\rho) - c(e)\right),$$
where $c(\cdot)$ denotes the Casimir number of a representation of $\son$ \cite{dmjcpgmh,pg-weyl},
defined as 
$c(\rho) = \langle \rho +\delta,\rho +\delta\rangle - \langle \delta,\delta\rangle$ where $\delta=(\delta_1,\dots,\delta_m)$ is the half-sum of
the roots, defined by $\delta_k = \tfrac{n-2k}{2}$ for $1\leq k \leq m$. Hence,
$c(e)=n-1$ for the standard representation $e$ of $\son$ on $\RM^n$. 
The main interest of conformal weights is that they are
easy to compute from the knowledge of the dominant weight of the representation \cite{dmjcpgmh}. Indeed,
\begin{equation}\label{eq:conformal-weights}
\begin{cases} 
w(\lambda,\rho) = \rho_i-i+1 & \textrm{if } \lambda = \rho + \mu_i \textrm{ for some } i,\\ 
w(\lambda,\rho) = -\rho_i+i-(n-1) & \textrm{if } \lambda = \rho - \mu_i \textrm{ for some } i,\\
w(\lambda,\rho) = -\tfrac{n-1}{2} & \textrm{if } \lambda=\rho.
\end{cases}
\end{equation}
The definition of the conformal weight operator extends to the bundle level, giving rise to a natural map
$B : T^*M\otimes E \rightarrow  T^*M\otimes E$ which reduces to multiplication by the relevant conformal weight on each 
subbundle $F_j$. This version of the conformal weight operator appears naturally when one looks at the behaviour of the 
Stein-Weiss operators $P_j$ with respect to conformal changes of metrics. Its eigenvalues can then be interpreted as the 
values of the powers of the conformal factor needed to make each of these operators conformally covariant, hence
their names \cite{fegan}.

Bochner-Weitzenb\"ock formulas (or, in short, Weitzenb\"ock formulas) are the last ingredients that are necessary for our analysis. They are defined as choices of sets of coefficients 
$(a_1, \dots,a_N)\in\RM^N$ such that the subsequent linear combination of the operators $(P_j)^*P_j$
is a zeroth-order operator, that is a curvature term:
\begin{equation}\label{weitzenbock} \sum_{j=1}^N a_j\, (P_j)^*P_j - \mathcal{R}=0.
\end{equation}
Bochner-Weitzenb\"ock formulas are classified. It is known that there are $\lfloor\tfrac{N}{2}\rfloor$ linearly independent
such formulas \cite{branson-jfa}, and finding them is a purely 
algebraic problem (although somehow tricky), which reduces 
to finding coefficients $(a_j)_{1\leq j\leq N}$ such that the principal symbol of the operator $\sum_{j=1}^N a_j\, (P_j)^*P_j$ has 
vanishing second-order part. They can be explicitly obtained through VanderMonde systems \cite{homma-casimir} or
recursive formulas \cite{semmelmann-weingart-weitzenbock}. 

\begin{assum}\label{assumption} 
From now on, coefficients $(a_j)$ are chosen such that they give rise to a Weitzenb\"ock formula as in (\ref{weitzenbock}).
\end{assum}

Letting $P_+ = \sum_{a_j>0}\sqrt{a_j}\,P_j$ and $P_- = \sum_{a_j<0}\sqrt{-a_j}\,P_j$, the Bochner-Weitzenb\"ock formula \eqref{weitzenbock} translates as
$$ (P_+)^*P_+ - (P_-)^*P_- - \mathcal{R}=0.$$
We now apply it to a section $\sigma$ of $E$ over an asymptotically flat manifold $(M,g)$. Integrating against $\sigma$ and performing an integration by parts, assuming momentarily that all terms
converge, we get an equality similar to (\ref{lichne_int}):
\begin{equation}\label{weitzenbock_int} 
\int_M |P_-\sigma|^2 + \langle\sigma,\mathcal{R}\sigma\rangle - |P_+\sigma|^2 
= \lim_{r\to\infty} \int_{S_r} b(\sigma) 
\end{equation}
where $b$ stands again for 
some boundary contribution to be described more explicitly later on.
This can be used towards a positive mass theorem if 
\begin{enumerate}[label=(\roman*)]
\item\label{item3} a section $\sigma$ of $E$ can be found such that $P_+\sigma=0$ and all integrals and limits make sense;
\item\label{item1} the curvature term $\mathcal{R}$ is nonnegative in the context at hand;
\item\label{item2} the boundary term can be related to the mass.
\end{enumerate}
Interestingly enough, this may also lead to less usual \emph{negative} mass theorems if $\mathcal{R}$ is nonpositive and $P_+$ is replaced by $P_-$ in \ref{item3}.

From  now on, our goal will be to give an answer to the question raised by condition \ref{item2}. As in Witten's proof, we shall consider here the case where
the section $\sigma$ is asymptotic to a constant section $\sigma_0$ of $E$.
We shall then show that the boundary contribution 
only depends on $\sigma_0$ if some natural asymptotic behaviour is assumed on $\sigma$. This fact opens the way to a refined study of the boundary contribution, which is a purely algebraic problem. 

We shall not discuss conditions \ref{item3} and \ref{item1}: condition \ref{item3} is usually obtained through the resolution of PDE problems which have been
extensively considered elsewhere, see \cite{ba,minerbe-alf-mass,parker-gauge,witten} for examples (note however that this may be avoided in some cases, as \cite{jammes} shows), and the validity of condition \ref{item1} of course 
depends on the context 
at hand.

To be more precise in our analysis of the question raised by condition \ref{item2},
we shall still need a bit more of notation. On our asymptotically flat manifold $(M,g)$,
we define a self-adjoint map $H$ of $T(M\setminus K)$ by 
$$ \langle HX,HY\rangle_g = \langle X,Y\rangle_0 \quad \forall\, X, Y \in T(M\setminus K) ,$$ 
where the latter scalar product is the Euclidean metric on $M\setminus K$ identified to $\RM^n\setminus B$. This enables us
to transfer sections of bundles over $\RM^n\setminus B$ to bundles over $M\setminus K$ (or \emph{vice-versa}) in a metric
preserving way~: this is of course obvious in the tensor case, and the spinor and mixed cases may be treated as in \cite{jpb-pg}.
Thus it makes sense to speak of \emph{constant sections} of the bundle $E$ w.r.t. to the flat connection over a neighbourhood of infinity. 
Equivalently, one may transfer the flat connection $\nabla^0$ itself to a flat metric connection (but with torsion) on the (spinor) 
frame bundle $\mathcal{P}$ or, conversely, the connection $\nabla$ induced by $g$ as a connection (again non torsion-free) on the 
trivial bundle $(\RM^n\setminus B)\times \mV$ which is compatible with the euclidean metric. As a result, if 
$\{e_i\}_{1\leq i\leq n}$ is a
(direct) orthonomal basis of $\RM^n$, if $\omega$ is the connection $1$-form of $\nabla$ in the frame $\{\varepsilon_i=He_i\}_{1\leq i\leq n}$,
and if $\sigma_0$ is a constant section of $E$ (w.r.t. to the flat connection), 
then $\nabla^0\sigma_0 =0$ and $\nabla\sigma_0 = \rho(\omega)\sigma_0$.

We now proceed to our study of the boundary contribution in \eqref{weitzenbock_int}. We now let once and for all $A = \sum_j a_j\Pi_j$, with $(a_1, \dots,a_N)$, chosen such that it gives rise to a Weitzenb\"ock formula, and we shall use the notations $P_+$ and $P_-$ introduced earlier. We also let $\sigma_0$ be a constant section of $E$ over the Euclidean space (transferred to the end of $M$ as above) and another section $\sigma = \sigma_0+\sigma_1$ where $\sigma_1$ tends to $0$ at infinity. 

Before stating our first technical result, we shall spend a few lines to justify the assumptions that will be made below on $\sigma_1$. A reasonable setting where condition \ref{item3} above can be obtained is when $P_+$ is an injectively elliptic operator (\emph{i.e} its principal symbol is injective for any nonzero covector, or equivalently $(P_+)^*P_+$ is elliptic in the usual sense). Examples include the Dirac operator or the Hodge-de Rham operators, see \cite{ba,minerbe-alf-mass,parker-taubes}. If $ \sigma$ is given such that $P_+\sigma = 0$, then
$$(P_+)^*P_+ \sigma_1 = - (P_+)^*P_+ \sigma_0  = - \left( (P_+^0)^*P^0_+ - (P_+)^*P_+\right)\sigma_0 ,$$
where $\cdot^{0}$ stands for the Euclidean metric. The analysis of the asymptotic behaviour of $\sigma_0$ can then be done in weighted H\"older spaces on
asymptotically flat manifolds, see \cite{ba,lockhart} for instance, and under the assumptions that $(P_+)^*P_+$ is elliptic, $\mathcal{R}$ is nonnegative and $\tau>\tfrac{n-2}{2}$, one may get that $\sigma_1 = O(r^{-a})$ and $\nabla\sigma_1 = O(r^{-a-1})$ for some $a>\tfrac{n-2}{2}$. 

We now come back to the case where $A = \sum_j a_j\Pi_j$ is a general projection 
with $(a_1, \dots,a_N)$, chosen such that it gives rise to a Weitzenb\"ock formula, and we state:

\begin{lemm}\label{lem-bord1}
Let $\sigma_0$ be an element of $\mathbb{V}$, seen as a constant section of $E$ over $\RM^n$. 
If $\sigma$ is a section of $E$ such that $|\sigma - \sigma_0| = O(r^{-a})$ 
and $\nabla\sigma = O(r^{-a-1})$ for some $a>\tfrac{n-2}{2}$, then the integrals in (\ref{weitzenbock_int}) converge and the limit of the boundary term only depends on the constant part of $\sigma$. More precisely,
$$ \lim_{r\to\infty}\int_{S_r} b(\sigma) \ = \ - \ \lim_{r\to\infty}\int_{S_r}
\langle \nu\otimes\sigma_0 , A(\rho(\omega)\sigma_0)\rangle \ d\!\vol_{S_r}  .$$
\end{lemm}

It is necessary to have $a>\tfrac{n-2}{2}$ for the arguments of the proof to hold true. It is remarkable that this decay rate is
the same as the decay rate usually inferred from a PDE analysis of the equations
$$ P_+\sigma = 0,\quad \sigma \to_{\infty} \sigma_0 $$
in weighted functional spaces, as mentioned above.

\smallskip

\begin{proof} From \eqref{weitzenbock_int}, one computes on a bounded domain $D$ that
$$ \int_D |P_+\sigma|^2 - |P_-\sigma|^2 \ = \ \sum_{j=1}^N a_j\, \int_D |\Pi_j(\nabla\sigma)|^2 \ = \ 
\sum_{j=1}^N a_j\, \int_D \langle\nabla\sigma,\Pi_j(\nabla\sigma)\rangle .$$
It is thus enough to manage the integration by parts for a single summand corresponding to an index $j$ in 
$\{1,...,N\}$: 
\begin{align*}
\int_D \langle\nabla\sigma,\Pi_j(\nabla\sigma)\rangle 
& =  \ \int_D \tr_g\left[\nabla\left( \langle\sigma,\Pi_j(\nabla\sigma)\rangle\right)\right] - 
\langle\sigma,\tr_g\left[\nabla\left(\Pi_j(\nabla\sigma)\right)\right]\rangle \\
& =  \ \int_{\partial D} \langle\sigma,\Pi_j(\nabla\sigma)\rangle(\nu) - \int_D 
\langle\sigma,\tr_g\left[\nabla\left(\Pi_j(\nabla\sigma)\right)\right]\rangle  \\
& =  \ \int_{\partial D} \langle\nu\otimes\sigma,\Pi_j(\nabla\sigma)\rangle + \int_D 
\langle\sigma,(P_j)^*P_j(\sigma)\rangle .
\end{align*}
Hence,
\begin{align*}
\int_D |P_+\sigma|^2 - |P_-\sigma|^2 \ = \int_{\partial D} \langle\nu\otimes\sigma,A(\nabla\sigma)\rangle + \int_D 
\langle\sigma,\mathcal{R}\sigma\rangle .
\end{align*}
Applying this to any domain enclosed by a coordinated sphere $S_r$, one gets:
$$ \int_{S_r} b(\sigma) \ = \ - \int_{S_r} \langle \nu\otimes\sigma , A(\nabla\sigma)\rangle \ d\!\vol_{S_r}  .$$
We now show that the limit as $r\to\infty$ only depends on $\sigma_0$. We again compute first with a single projection.
One writes $\sigma=\sigma_0+\sigma_1$, where $|\sigma_1| = O(r^{-a})$ and $\nabla\sigma_1 = O(r^{-a-1})$, and
one computes on a sphere $S_r$~:
\begin{align*} \int_{S_r} \langle \nu\otimes\sigma , \Pi(\nabla\sigma)\rangle \ = & \ 
\int_{S_r} \langle \nu\otimes\sigma_0 , \Pi(\nabla\sigma_0)\rangle +
\int_{S_r} \langle \nu\otimes\sigma_0 , \Pi(\nabla\sigma_1)\rangle \\ 
& \ \ \ +
\int_{S_r} \langle \nu\otimes\sigma_1 , \Pi(\nabla\sigma_0)\rangle +
\int_{S_r} \langle \nu\otimes\sigma_1 , \Pi(\nabla\sigma_1)\rangle .
\end{align*}
When $r\to\infty$, the assumptions on $\sigma_1$ imply that the last two terms vanish when
$\tau>\tfrac{n-2}{2}$ as their integrands are $o(r^{-(n-1)})$ whereas the volumes of $S_r$ grow as $r^{n-1}$. 
As regards the second term, one notices that (denoting with a bold dot an absent variable)
\begin{align*} \langle \bullet\otimes\sigma_0 , \Pi(\nabla\sigma_1)\rangle \ 
& 
= \ \sum_{j=1}^N \langle \Pi(\bullet\otimes\sigma_0) , \varepsilon_j\otimes\nabla_{\varepsilon_j}\sigma_1\rangle \\
& = \ \sum_{j=1}^N \, \nabla_{\varepsilon_j}\left( \langle \Pi(\bullet\otimes\sigma_0) ,\varepsilon_j\otimes\sigma_1\rangle\right)\\
& \ \ \ \ \ \ \ \ \ \ \ \ \ \ \ \ 
- \sum_{j=1}^N \,  \langle \Pi(\bullet\otimes\nabla_{\varepsilon_j}\sigma_0) , \varepsilon_j\otimes\sigma_1\rangle ,
\end{align*}
where $\{\varepsilon_i \}_{1\leq i\leq N} = \{He_i\}_{1\leq i\leq N}$ is the $g$-orthonormal frame of $TM$ deduced from a (constant) Euclidean
basis of $\RM^n$. After integrating on $S_r$ and letting $r$ tend to
$\infty$, the same `decay \emph{vs.}\! volume growth' considerations as above show that the very last term does not contribute and
it remains to consider the first one. Coming now back to the case at hand, where $A$ is in the picture, the precise term we need to study is
$$ \sum_{j=1}^N \, \nabla_{\varepsilon_j}\left( \langle A(\bullet\otimes\sigma_0) ,\varepsilon_j\otimes\sigma_1\rangle\right), $$
but this is the divergence of the bilinear form $\theta$ defined by
$$ \theta (X,Y) = \langle A(X\otimes\sigma_0) ,Y\otimes\sigma_1\rangle .$$
We now use the fact that the $N$-uple 
$(a_1,\dots,a_N)$ gives rise to a Weitzenb\"ock formula, hence the second-order part of the principal symbol of $\sum a_j(P_j)^*P_j$ vanishes. Let us denote this second-order  part by $q_{\xi}$ on any covector $\xi$ (forgetting for a moment that it vanishes). It is easily computed that
$$ \langle q_{\xi}(\sigma_0),\sigma_1\rangle \ = \ \langle A(\xi \otimes \sigma_0), \xi\otimes\sigma_1\rangle \ = \ \theta(\xi,\xi), $$
and this expression is equal to $0$ since $q_{\xi}$ vanishes.
Thus $\theta$ is a $2$-form, being equal to its antisymmetric part.
Coming back to our previous computation, the term we have to integrate is of the form $*d*\theta$ where $*$ is
the ($n$-dimensional) Hodge operator. Thus, on any closed hypersurface $S$ with outer unit normal $\nu$,
$$ \int_{S} (*\, d*\theta)(\nu) \, d\!\vol_{S} \ = \ \int_{S}  d*\theta \ =  \ 0.$$ 
This shows that the only nonzero contribution in the limiting boundary term comes from the term containing only $\sigma_0$ above, as expected.
\end{proof}

\section{Main statement}\label{sec:main}

We now proceed to the main part of our paper. Recall that we consider a natural bundle $E$ over $(M,g)$ issued from an irreducible representation $(\mV,\rho)$ of the special orthogonal group $\Spon$ or the spinor group $\Spin(n)$. The tensor product $\RM^n\otimes\mV$ splits into irreducibles as $\mW_1\oplus\cdots\oplus\mW_N$, with projections $\Pi_j$, and we have chosen a set $\mathbf{a}=(a_j)_{1\leq j\leq N}$ such that $\sum a_j(P_j)^*P_j$ is a curvature operator $\mathcal{R}$, where $P_j = \Pi_j\circ\nabla$, see \eqref{weitzenbock}.

Lemma \ref{lem-bord1} can then be understood as follows: the boundary-at-infinity contribution of the integrated version of the identity $\sum a_j(P_j)^*P_j=\mathcal{R}$ can be seen as a map
$$ \sigma \mapsto \lim_{r\to\infty}\int_{S_r} b(\sigma) ,$$
acting on sections $\sigma$ that are constant at infinity (with asymptotic behaviour as in Lemma \ref{lem-bord1}), and this map is a quadratic form over the boundary values, \emph{i.e} the constant sections:
$$ \sigma_0 \in \mV \mapsto \mathcal{Q}(\sigma_0) = - \lim_{r\to\infty}\int_{S_r}
\langle \nu\otimes\sigma_0 , A(\rho(\omega)\sigma_0)\rangle \ d\!\vol_{S_r}  $$
where $\omega$ is the connection $1$-form of the Levi-Civita connection of $g$. (Note that one rather has a Hermitian form if $\mV$ is a complex space, but we shall continue to call it `quadratic' as this will be a harmless modification). We shall now prove that the \emph{trace} of this quadratic form is always related to the mass. Our
 main result is:
 
\begin{theo}\label{theo:main}
Let $(\mV,\rho,E)$ be as above, and $\mathbf{a}=(a_j)_{1\leq j\leq N}$ be a choice of coefficients giving rise to a 
Weitzenb\"ock formula as in \eqref{weitzenbock}. Then there is a constant $\mu(\mathbf{a})$ such that, given any orthonormal basis $(\sigma_0^{\kappa})_{1\leq \kappa\leq \dim \mV}$ of $\mV$,
$$ - \, \lim_{r\to\infty}\int_{S_r} \, \sum_{\kappa=1}^{\dim\mV} \, 
 \langle \,\nu\otimes\sigma_0^{\kappa} \, , \, A(\rho(\omega)\sigma_0^{\kappa})\,\rangle_0 \ d\!\vol_{S_r} 
\ = \ \mu(\mathbf{a}) \, m(g).$$
Moreover,  $\mu(\mathbf{a})$ only depends on $\mathbf{a}$ and 
representation-theoretic data of $\mV$:
$$ \mu(\mathbf{a}) \ = \ -\,\sum_{j=1}^{N}\, a_j\, \frac{(\dim\mW_j)\, w(\lambda_j,\rho)}{2n(n-1)}\, ,$$
where $w(\lambda_j,\rho)$ denotes the \emph{conformal weight} of the summand $\mW_j$ in $\RM^n\otimes \mV$.
\end{theo}

The content of this theorem is as follows: although it is not always true that the mass appears in the boundary-at-infinity contribution in \eqref{weitzenbock_int} for a single section, it does always appear when we sum the formula over a basis of constant sections. Coming back to \eqref{weitzenbock_int}, one has:
$$ \mu(\mathbf{a}) \, m(g) \ = \sum_{\kappa=1}^{\dim\mV} \mathcal{Q} (\sigma_0^{\kappa}) \ = \ \sum_{\kappa=1}^{\dim\mV} \ \int_M \, |P_-\sigma^{\kappa}|^2 - |P_+\sigma^{\kappa}|^2 + \langle\sigma^{\kappa},\mathcal{R}\sigma^{\kappa}\rangle ,$$
whenever $\sigma^{\kappa} \rightarrow_{\infty} \sigma_0^{\kappa}$ for each $\kappa=1,\dots,\dim\mV$ with the right asymptotic behaviour. Assuming moreover that $P_{\pm}\sigma^{\kappa} = 0$ for at least one choice of sign and for each $\kappa=1,\dots,\dim\mV$, one may obtain \emph{positive} or \emph{negative} mass theorems, depending on the combination of signs of both $\mu(\mathbf{a})$ and the curvature operator $\mathcal{R}$. 

The proof of Theorem \ref{theo:main} is divided into two parts: existence of $\mu(\mathbf{a})$, then its computation. The first step is very simple if the dimension $n$ is not equal to $4$. Indeed, one notices 
that the map $\beta_A$ from $\RM^n\otimes\son$ into $\RM^n$ defined by
$$ \omega \mapsto  \beta_A(\omega) \, = \,  
- \, \sum_{\kappa=1}^{\dim\mV}\, \langle \,\cdot\otimes\sigma_0^{\kappa}, A(\rho(\omega)\sigma_0^{\kappa})\,\rangle_0 $$
is equivariant under the action of $\son$. (If $\mV$ is complex, one should rather take the real part of this map, since the boundary contribution resulting from the integration by parts should always be real.) 

If $n\neq 4$, there is always a unique factor $\RM^n$ in the decomposition of $\RM^n\otimes\son$ into irreducible summands (this is a simple computation that can be done using the rules recalled in section \ref{sec:setting} and that we shall leave to the reader). Thus Schur's Lemma implies that
$\beta_A$ is always a multiple of the orthogonal projection onto this unique summand, identified to $\RM^n$. Letting $\pi$ denote this projection, this means that
$$ \beta_A(\omega) \, = \,  \textrm{\emph{const}}(\mV,\rho,\mathbf{a})\,\pi $$
where $\textrm{\emph{const}}(\mV,\rho,\mathbf{a})$ denotes a constant depending only on the choice of the representation $(\mV,\rho)$ and the coefficients $\mathbf{a}=(a_j)$.
Moreover, the computation of this map is well understood in the special case where $(\mV,\rho)$ is the spin representation, $\mathbf{a}=(a_j)$ is the set of coefficients leading to the classical Lichnerowicz-Schr\"odinger formula, and $\omega$ comes from the Levi-Civita connection of the Riemannian metric $g$. In this case, $\beta_A(\omega)$ has the same dominant term at infinity as a (precisely known) multiple of $\diver_0 g - d(\tr_0 g)$.
Thus, once integrated, the (traced) boundary contribution has the same limit as
$$  \int_{S_r}\,\left( \diver_0 g - d(\tr_0 g)\right)(\nu_r)\ d\!\vol_{S_r}  ,$$ 
which itself converges to a positive multiple of the mass as $r$ tends to infinity.
This then shows that the limit of the (traced) boundary contribution is a multiple of the mass for any choice of representation $(\mV,\rho)$ and coefficients $\mathbf{a}=(a_j)$. This in turn proves the existence of the constants $\mu(\mathbf{a})$ if $n\neq 4$. 

The case $n=4$ is slightly more involved as two distinct copies of $\RM^4$ appear in the decomposition of $\RM^4\otimes\mathfrak{so}(4)$ into irreducibles, due to the self-duality phenomenon. Thus, the previous proof does not apply. We shall provide below an alternative argument, which relies on the precise study of the map $\beta_A$. As this study is also necessary for the computations of the values of the constants $\mu(\mathbf{a})$, this will be done in the next section.

\section{Computation for a single projection}\label{sec:computation}

We now proceed to the precise computation of $\mu(\mathbf{a})$. As the sequel will show, the missing arguments in the $n=4$ case will also come as a by-product of some of our computations below. 

We fix here a given factor $(\mW,\lambda)$ appearing in the decomposition of $\RM^n\otimes\mV$ into irreducibles. As $\mu(\mathbf{a})$ is obviously linear in $\mathbf{a}$, it is enough to compute the constant when $A$ is a  projection on a single irreducible summand $\mW$ in $\RM^n\otimes\mV$. We thus let $A = \Pi$, the projection onto $\mW$, seen as a subspace of $\RM^n\otimes\mV$. Thus, the map to be considered is $\beta = \beta_{\Pi}$ defined by
$$ \beta(\omega) \, = \,  
- \, \sum_{\kappa=1}^{\dim\mV}\, \langle \,\cdot\otimes\sigma_0^{\kappa}, \Pi(\rho(\omega)\sigma_0^{\kappa})\,\rangle_0 $$
where $(\sigma_0^{\kappa})_{1\leq \kappa\leq \dim \mV}$ is any orthonormal basis of $\mV$.
(To be precise enough, we should recall that we must take the real part of this map if $\mV$ is complex: as the sequel will show, we can forget about the real part as the result will turn out to be real.) 

\subsection*{Preliminaries and the proof when $n\neq 4$}
To compute $\beta$, one may restrict it to the subfactor $\RM^n$ in $\RM^n\otimes\son$ induced by the equivariant injection of the former into the latter that is explicitly described as follows: for any $1$-form $\alpha$, one lets $i(\alpha)$ be the $2$-form with values
into $\RM^n$, \emph{i.e.} an element of $\RM^n\otimes\son$, defined by
$$ i(\alpha)\, (X,Y) \ = \ (\alpha\wedge I)\,(X,Y) \ = \ \alpha(X)Y - \alpha(Y)X . $$
Our goal is now to compute $\beta\circ i$.
As a first step, we study $\rho(\alpha\wedge I)$ for an arbitrary $1$-form $\alpha$ (recall that we freely identify elements
of $\RM^n$ and $(\RM^n)^*$, \emph{i.e.} forms and vectors). By definition, for any vector $Z$ in $\RM^n$,
\begin{align*} 
(\alpha\wedge I)(Z) & = \sum_{i,j} \left( \alpha(e_i)\langle e_j, Z\rangle - \alpha(e_j)\langle e_i, Z\rangle \right) e_i\otimes e_j \\
& = \sum_{i<j} \left( \alpha(e_i)\langle e_j, Z\rangle - \alpha(e_j)\langle e_i, Z\rangle \right) e_i\wedge e_j .
\end{align*}
or equivalently,
\begin{equation*} 
(\alpha\wedge I) =  \sum_{i<j}\left( \alpha(e_i) e_j - \alpha(e_j) e_i\right) \otimes  e_i\wedge e_j  .
\end{equation*}
This leads to:
\begin{align*} 
\rho(\alpha\wedge I) & =  \sum_{i<j} \left( \alpha(e_i) e_j - \alpha(e_j) e_i\right) \otimes \rho(e_i\wedge e_j) \\
& = \sum_{i<j} \alpha(e_i) e_j \otimes \rho(e_i\wedge e_j) - \alpha(e_j) e_i \otimes \rho(e_i\wedge e_j) \\
& = \sum_{i,j} \alpha(e_i) e_j \otimes \rho(e_i\wedge e_j) \\
& = \sum_{j} e_j \otimes \rho(\alpha\wedge e_j) .
\end{align*}
This is nothing but the opposite of the conformal weight operator! More precisely, 
$$\rho(\alpha\wedge I)\sigma = -B(\alpha\otimes\sigma)\ \ \textrm{for any } \sigma \in \mV. $$
These remarks reduce our problem to computing, for any $\alpha$ in $\RM^n$, 
$$ \beta \circ i (\alpha) \, = \, w(\lambda,\rho)\,
 \sum_{\kappa=1}^{\dim\mV}\,\langle \,\cdot\otimes\sigma_0^{\kappa}, \Pi(\alpha\otimes\sigma_0^{\kappa})\,\rangle_0 \, .$$
But the map 
$$ \alpha\otimes X \in \RM^n\otimes\RM^n \ \mapsto \ 
\sum_{\kappa=1}^{\dim\mV}\,\langle \, X\otimes\sigma_0^{\kappa}, \Pi(\alpha\otimes\sigma_0^{\kappa})\,\rangle_0 \,\in\RM  $$
is invariant under the action of $\son$ on $\RM^n\otimes\RM^n$, thus must be a constant multiple of the 
trivial contraction $\alpha\otimes X \mapsto \alpha(X)$. It then suffices to compute
$$ \sum_{i=1}^n\sum_{\kappa=1}^{\dim\mV}\,\langle \, e_i\otimes\sigma_0^{\kappa}, \Pi(e_i\otimes\sigma_0^{\kappa})\,\rangle_0 \, .$$
This is of course the trace of the matrix representing $\Pi :\RM^n\otimes\mV \rightarrow \RM^n\otimes\mV$ in the 
orthonormal basis $\{ e_i\otimes\sigma_0^{\kappa}\}_{i,\kappa}$. As $\Pi$ is an orthogonal projection operator, 
this trace equals the rank of $\Pi$, \emph{i.e.} the dimension of its image $\mW$. Thus, for any $\alpha$, 
$$ \beta \circ i (\alpha) \, = \, w(\lambda,\rho)\,\sum_{\kappa=1}^{\dim\mV}\,\langle \,\cdot\otimes\sigma_0^{\kappa}, 
\Pi(\alpha\otimes\sigma_0^{\kappa})\,\rangle_0 \ = \ \frac{(\dim\mW)\, w(\lambda,\rho)}{n}\,\alpha \, .$$

We now assume that $n\neq 4$ and we denote by $\pi$ the map from $\RM^n\otimes\son$ onto $\RM^n$, seen as the unique summand isomorphic to
$\RM^n$ appearing in the decomposition of $\RM^n\otimes\son$ into irreducibles, such that $\pi \circ i (\alpha) =\alpha$ for any $\alpha\in\RM^n$. Thus,
\begin{equation}\label{eq:beta_Pi_vs_pi} 
\beta \ = \ \frac{(\dim\mW)\, w(\lambda,\rho)}{n}\ \pi\, .
\end{equation}
In the setting of Theorem \ref{theo:main}, where we consider no more a single projection but a linear combination
$A = \sum a_j\Pi_j$, this will imply that
\begin{equation}\label{eq:beta_A_vs_pi}
\beta_A \ = \ \left(\sum_{j=1}^{N}\, a_j\, \frac{(\dim\mW_j)\, w(\lambda_j,\rho)}{n}\right)\,\pi\, . 
\end{equation}

It now remains to relate this to the mass. For this, we need a few more computations, that can be done for any dimension. Thus we shall drop here the assumption that $n\neq 4$. Let $(\mV,\rho)$ be the spinor representation $(\Sigma,\varsigma)$ in odd dimensions, and the positive half-spinor representation $(\Sigma_+,\varsigma_+)$ in even dimensions. In both cases, there are only two irreducible summands in the decomposition of the tensor product into irreducibles: one has
$\RM^n\otimes\Sigma = \TM\oplus\Sigma$ in odd dimensions and $\RM^n\otimes\Sigma_+ = \TM_+\oplus\Sigma_-$ in even dimensions, where $\TM$ and $\TM_{+}$ denote the twistor and positive half-twistor representations, and $\Sigma_-$ is the negative half-spinor representation.
 
To have uniform notations regardless of dimensions, we shall now denote by $\mV$ the relevant spinor or half-spinor representation. One may then write $\RM^n\otimes\mV = \mW_1\oplus\mW_2$, where $\mW_1$ is the relevant twistor representation and $\mW_2$ 
the relevant spinor representation appearing in this decomposition. 
 We now consider the endomorphism $A$ associated to this situation, \emph{i.e.}
$$a_1=-1, \quad a_2=(n-1), \quad A = (n-1)\Pi_2 - \Pi_1, $$ 
so that the Main Assumption of page \pageref{assumption} is satisfied. The classical Lichnerowicz formula 
$\mathfrak{D}^*\mathfrak{D} - \nabla^*\nabla = \tfrac{1}{4}\Scal$ can then be rewritten as
$$ (n-1) (P_2)^*P_2 - (P_1)^*P_1 = \tfrac{1}{4}\Scal \, .$$
(Note that one has to be careful with the choice of norms: here and everywhere else in the paper, each $\mW_j$ is endowed with 
the norm induced by the product norm on $\RM^n\otimes\mV$; as a result, $|\Pi_2(\nabla\psi)|^2 = \tfrac{1}{n}|\mathfrak{D}\psi|^2$
for any spinor field~$\psi$.)
We shall now re-interpret Witten's boundary term in our setting.
It has indeed been computed in \cite{andersson-dahl-alh} that for any $j$, $k$ in $\{1,...,n\}$ and any tangent vector
$X$,
\begin{align*}
\omega_j^k(X) \ & =  \ \ \frac12\, \langle\, (\nabla^0_XH)e_j - (\nabla^0_{He_j}H)H^{-1}X\, ,\, He_k \,\rangle \\
&  \ \ \ \ - \frac12\,\langle\, (\nabla^0_XH)e_k - (\nabla^0_{He_k}H)H^{-1}X\, ,\, He_j \,\rangle \\
&  \ \ \ \ - \frac12\,\langle\, (\nabla^0_{He_j}H)e_k - (\nabla^0_{He_k}H)e_j\, ,\, HX \,\rangle  .
\end{align*}
The decay assumption $\tau>\tfrac{n-2}{2}$ moreover implies that when passing to the limit $r\to\infty$,
the metric $\langle \cdot , \cdot \rangle$ may be replaced by  $\langle \cdot , \cdot \rangle_0$, all 
occurrences of $H$ at zeroth-order may be replaced by the identity, and $\nabla^0H$ may be replaced by $-\tfrac12\nabla^0g$, 
all of this without harm. More precisely, we shall use below the notation $U\simeq V$ to mean that $V$ is the only term in the expression of $U$ that contributes when integrating $U$ on larger and larger spheres. Equivalently, 
$$ U \simeq V \ \Leftrightarrow \ \lim_{r\to\infty} \int_{S_r} U = \lim_{r\to\infty} \int_{S_r} V. $$
Thus, 
\begin{align*}
\omega_j^k(X) \ & \simeq  \ \ \frac14\, \langle\, (\nabla^0_Xg)e_k - (\nabla^0_{e_k}g)X\, ,\, e_j \,\rangle_0 \\
&  \ \ \ \ - \frac12\,\langle\, (\nabla^0_Xg)e_j - (\nabla^0_{e_j}g)X\, ,\, e_k \,\rangle_0 \\
&  \ \ \ \ - \frac12\,\langle\, (\nabla^0_{e_k}g)e_j - (\nabla^0_{e_j}g)e_k\, ,\, X \,\rangle_0  .
\end{align*}
Taking into account the symmetry of $g$, this reduces to 
\begin{equation}\label{eq:gamma_de_g1}
\omega_j^k(X) \ \simeq \ \frac12\,\left[ \nabla^0_{e_j}g (X,e_k) \, - \, \nabla^0_{e_k}g (X,e_j)  \right] ,
\end{equation}
where the right-hand side can be equivalently described as the $2$-form-valued $1$-form 
\begin{equation}\label{eq:gamma_de_g2}
(X,Y,Z) \ \mapsto \ \frac12\,\left[ \nabla^0_{Y}g (X,Z) \, - \, \nabla^0_{Z}g (X,Y)  \right] .
\end{equation}
This formula, which is valid in any dimension, will also be the key element needed to fill the remaining gap in dimension $4$, a step that will be done at the end of this section.

We first handle the case when $n\neq 4$. As explained above, $\beta_A(\omega)$ is a multiple of the projection of $\omega$ onto the irreducible summand $\RM^n$ in $\RM^n\otimes\son$. The identification of this summand with $\RM^n$ can be obtained through the trace of $\omega$ in the variables $X$ and $Y$ in \eqref{eq:gamma_de_g2}. From this, one sees that 
$$ \beta_A(\omega) \ \simeq \ \textrm{\emph{const.}} \, (\diver_0 g - d\tr_0 g),$$ 
the right-hand side being the expected term in the boundary contribution as it also appears in the definition of the mass. In Witten's case, the proportionality
factor is known precisely (see section \ref{sec:witten}): the integrand of the boundary term now comes as
$$ \beta_A(\omega) \ \simeq \ \frac{\dim \Sigma_{(+)}}{4}\, \left( \diver_0 g - d \tr_0 g \right) ,$$
where we have used the notation $\Sigma_{(+)}$ to denote the spinor representation space $\Sigma$ in odd dimensions and the positive half-spinor representation space $\Sigma_+$ in even dimensions. Note also that we have taken a trace over an orthonormal basis of the relevant spinor bundle, so that $\dim \Sigma_{(+)}$ appears in
this formula.
Of course, this can be compared with the expression obtained before, \emph{i.e.}
$$ \beta_A(\omega) \ =\  \left( \frac{(n-1)}{n}\,(\dim\mW_2)\, w(\lambda_2,\varsigma_{(+)})\, 
- \,\frac{1}{n}\,(\dim\mW_1)\, w(\lambda_1,\varsigma_{(+)})\right)\pi(\omega) ,$$
where the same convention as above has been used for the notation $\varsigma_{(+)}$.

One has $\dim \Sigma = 2^{\lfloor\tfrac{n}{2}\rfloor}$ in odd dimensions and $\dim \Sigma_+ = \dim \Sigma_- = 2^{\tfrac{n}{2}-1}$ 
in even dimensions, and thus,
$$\dim \mW_1 = 2^{\lfloor\tfrac{n}{2}\rfloor-\delta}(n-1), \quad \dim \mW_2 = 2^{\lfloor\tfrac{n}{2}\rfloor-\delta}$$ 
where $\delta=0$ in odd dimensions and $\delta=1$ in even dimensions. It is moreover easily computed that, irrespective
of the parity of dimension,
$$ w(\varsigma,\varsigma) = w(\varsigma_-,\varsigma_+) =  -\frac{n-1}{2}, \quad 
w(\mathfrak{t},\varsigma) = w(\mathfrak{t}_+,\varsigma_+) = \frac12 ,$$
where $\mathfrak{t}_{(+)}$ is the relevant twistor representation (same convention as above). Thus,
\begin{align*}
\beta_A(\omega) \ & = \ - \, \left( \frac{(n-1)^2}{n}\,(\dim\mW_2)\, + \,\frac{1}{2n}\,(\dim\mW_1)\right)\, \pi(\omega) \\
& = \ - \,\frac{(n-1)}{2}\, 2^{\lfloor\tfrac{n}{2}\rfloor-\delta} \,\pi(\omega) \\
& = \ - \,\frac{(n-1)}{2}\,\left(\dim \Sigma_{(+)}\right) \,\pi(\omega) .
\end{align*} 
From this, one deduces immediately that
$$ \pi(\omega)\ \simeq -  \frac{1}{2(n-1)}\, \left( \diver_0 g - d \tr_0 g \right) .$$
Coming back to the case of a single projection $\Pi$ onto a summand $\mW$ of $\RM^n\otimes\mV$ with
$\mV$ being an arbitrary irreducible representation of $\son$, one concludes that
$$ \beta(\omega) \ \simeq \ -\, \frac{(\dim\mW)\, w(\lambda,\rho)}{2n(n-1)}\, \left( \diver_0 g - d \tr_0 g \right) , $$
and this ends the proof of Theorem \ref{theo:main} when $n\neq 4$.\qed

\smallskip

\subsection*{The proof in dimension $4$}
We shall now give the necessary arguments to get a complete proof in the missing four-dimensional case.
In this case, 
\begin{equation}\label{eq:dim4-deuxcopies1} 
\RM^4\otimes\Lambda^2\RM^4 = \mathbb{U}_+ \oplus \mathbb{U}_- \oplus \RM^4 \oplus \RM^4 
\end{equation}
where $\mathbb{U}_\pm$ are the irreducible representations whose dominant weights are $(2,\pm 1)$. Thus there are two copies of $\RM^4$ in $\RM^4\otimes\Lambda^2\RM^4$, the first 
one coming from the decomposition of $\RM^4\otimes\Lambda^2_+\RM^4$ and the second one coming from the decomposition of $\RM^4\otimes\Lambda^2_-\RM^4$ into irreducibles. It is thus not possible to conclude that $\beta_A$ must always be a multiple of a single projection as above. Indeed, there is a $2$-dimensional space of possible choices for the latter and one may only conclude that $\beta_A$ belongs to this space. One can be slightly more precise: if $(X,Y,Z) \mapsto \omega(X;Y,Z)$ is a generic element in $\RM^4\otimes\Lambda^2\RM^4$, then the maps 
$$\omega \mapsto T(\omega) = \tr_{12}(\omega) \ \textrm{ and } \ \omega \mapsto T'(\omega) = \tr_{12}(\star\omega) $$ 
provide a basis of the space under consideration. (Here $\tr_{12}$ denotes the trace in the $X,Y$-variables and $\star$ is the Hodge-star operator acting on the $2$-form part of elements in $\RM^4\otimes\Lambda^2\RM^4$, so that both $T$ and $T'$ are maps from $\RM^4\otimes\Lambda^2\RM^4$ to $\RM^4$.) As a result, one may write
\begin{equation}\label{eq:dim4-deuxcopies2} 
\RM^4\otimes\Lambda^2\RM^4 = \mathbb{U}_+ \oplus \mathbb{U}_- \oplus \mathbb{U}_0 \oplus \mathbb{U}_0'\ , 
\end{equation}
where $\mathbb{U}_0$, resp. $\mathbb{U}_0'$, is isomorphic to $\Im T$, resp. $\Im T'$, the latter two being copies of $\RM^4$ sitting diagonally in the last direct sum appearing in the r.h.s. of \eqref{eq:dim4-deuxcopies1}.

At this stage, it is important to notice that we don't need a general expression for $\beta_A(\omega)$ as we only want to compute it when $\omega \simeq \gamma(\nabla^0 g)$, where $\gamma$ is the map from $\RM^4\otimes\textrm{Sym}^2\RM^4$ to $\RM^4\otimes\Lambda^2\RM^4$ defined by 
$$ \gamma(\psi) (X;Y,Z) = \frac12 \left[ \psi(Y;X,Z) - \psi (Z;X,Y)\right] ,$$
see Equations (\ref{eq:beta_Pi_vs_pi}--\ref{eq:beta_A_vs_pi}). We can moreover write
$$ \RM^4\otimes\textrm{Sym}^2\RM^4 = \RM^4 \oplus \left(\RM^4\otimes\textrm{Sym}^2_0\RM^4\right).$$ 
Letting $\mathbb{S}$ be the irreducible representation of $\mathfrak{so}(4)$ whose dominant weight is $(3,0)$, the last factor in the previous formula decomposes as 
$$\RM^4\otimes\textrm{Sym}^2_0\RM^4 = \mathbb{S} \oplus \mathbb{U}_+ \oplus \mathbb{U}_- \oplus \RM^4,$$ 
and there are again two copies of $\RM^4$ inside $\RM^4\otimes\textrm{Sym}^2\RM^4$. 

But it turns out that $T'\circ\gamma$ is the zero map! Let us compute this on a decomposed element $h = u\otimes(v\otimes v)$ in $\RM^4\otimes\textrm{Sym}^2\RM^4$. It is obvious that $\gamma(h)=0$ if $u$ and $v$ are colinear, hence we may assume that $u$ is orthogonal to $v$. Thus, $\gamma(h) = v\otimes (u\wedge v)$ and $\star\gamma(h) = v\otimes \star (u\wedge v)$. As a result, 
$$T'\circ\gamma(h)=i_v\lrcorner\star(u\wedge v)=0.$$
This implies that the restriction of $\beta$ to $\Im\gamma$ is a multiple of a single projection, this time onto $\mathbb{U}_0$. Moreover, and since the identification of $\mathbb{U}_0$ with $\RM^4$ is done with the trace $T$, it is still true in dimension $4$ that
$$ \beta_A(\omega) \ \simeq \ \textrm{\emph{const.}} \, (\diver_0 g - d\tr_0 g)$$ 
for any $\omega \simeq \gamma(\nabla^0g)$. We now notice that $i(\alpha) = 2\, \gamma(\alpha\otimes\textrm{Id})$ for any $\alpha$, so that the image of $i$ also sits in $\mathbb{U}_0$. We may then define $\pi$ as the unique map from $\RM^4\otimes\Lambda^2\RM^4$ to $\RM^4$ which factors through the projection onto $\mathbb{U}_0$ and such that $\pi\circ i(\alpha) =\alpha$ for any $\alpha\in(\RM^n)^*$. Thus, similarly to the $n\neq 4$ case, 
$$\beta_{|\Im\gamma} = \textrm{\emph{const.}} \, \pi , $$
and we can proceed as above. Equations (\ref{eq:beta_Pi_vs_pi}--\ref{eq:beta_A_vs_pi}) are still valid if $n=4$ if one restricts $\beta$, resp. $\beta_A$, to the image of $\gamma$, and one may use the spinor case again to identify the value of the constant without changing any line. This ends the proof in dimension $4$.  \qed

\section{Examples of applications}\label{sec:appli}

We give here a choice of examples where the ideas above can be applied. Of course the
statements given here do not exhaust all possible 
applications. They however illustrate the fact that Theorem \ref{theo:main} is easy to use.

\subsection{A universal positive mass theorem} 
It has been remarked by Gau\-duchon \cite{pg-weyl} that
choosing $a_j = w(\lambda_j,\rho)$ (with notations as in the previous sections) always leads to a Weitzenb\"ock formula.
As a matter of fact, a trivial computation shows that the principal symbol $\sigma_{\xi}\left(\sum w(\lambda_j,\rho) (P_j)^*P_j\right)$
vanishes. Indeed,
$$\left\langle \sigma_{\xi}\left(\sum w(\lambda_j,\rho) (P_j)^*P_j\right)x\, , \,y\right\rangle 
= \langle B(\xi\otimes x),\xi\otimes y\rangle =0,$$ 
hence 
$\sum w(\lambda_j,\rho) (P_j)^*P_j$ is a zeroth-order operator. This formula is called the \emph{universal Weitzenb\"ock formula}
in \cite[Proposition 3.6, p. 519]{semmelmann-weingart-weitzenbock}.
One may now apply our main result to this case, and conclude that there is always an underlying universal positive mass theorem:

\begin{theo}\label{prop:universal}
Let $E$ be the natural bundle over a complete as\-ymptotic\-ally flat Riemannian manifold $(M,g)$
induced from an irreducible representation $(\mV,\rho)$ of $\son$. 
With the notations of section~\ref{sec:setting}, assume that 
the curvature operator 
$$ \mathcal{R} = \sum_{j=1}^N w(\lambda_j,\rho) (P_j)^*P_j $$
is nonpositive and that there exists a full set of solutions $(\sigma^{\kappa})_{1\leq \kappa\leq \dim\mV}$ 
of the equation 
$$ \sum_{w(\lambda_j,\rho)<0} w(\lambda_j,\rho) P_j\sigma^{\kappa} = 0 $$
that are asymptotic to an orthonomal basis $(\sigma_0^{\kappa})_{1\leq\kappa\leq\dim\mV}$ of $\mV$ and satisfy the estimates of Lemma \ref{lem-bord1}.
Then one has
$$ \sum_{\kappa=1}^{\dim\mV} \int_M \sum_{w(\lambda_j,\rho)> 0}\!\!\!\!\! w(\lambda_j,\rho)\, |P_j\sigma^{\kappa}|^2 
- \langle \sigma^{\kappa},\mathcal{R}\sigma^{\kappa}\rangle = 
\frac{c(\rho)}{n(n-1)}\,(\dim\mV)\, m(g),$$
where $c(\rho)$ is the Casimir operator of the representation $\rho$.
Hence the mass $m(g)$ is nonnegative.
\end{theo}

\begin{proof} The only missing point is the precise computation of the proportionality factor, which is \emph{a priori}
equal to 
$$ \frac{1}{2n(n-1)}\,\left(\sum_{j=1}^{N} w(\lambda_j,\rho)^2\,\dim\mW_j\right) .$$
The term inside brackets 
is obviously $\tr B^2$, and one may then apply \cite[\S 4, p.230]{dmjcpgmh}. We denote as in this paper 
$\ptr$ the partial trace of an operator from $\RM^n\otimes\mV$ to itself, \emph{i.e.} the endomorphism of $\mV$ obtained by
taking the trace 
on the $\RM^n$-factor. Then $\ptr(B^2)$ is nothing but twice the Casimir operator $c(\rho)$
of $(\mV,\rho)$ \cite[Equation~(4.1)]{dmjcpgmh}, and
$$  \sum_{j=1}^{N} w(\lambda_j,\rho)^2\,\dim\mW_j \ = \ \tr B^2 \ = \ (\dim\mV)\, \ptr(B^2) \ = \ 2\, (\dim\mV)\, c(\rho) ,$$
which proves the desired formula
\end{proof}

\subsection{The case when $N=2$} 
From \cite{dmjcpgmh,homma-casimir}, the number $N$ of irreducible summands in $\RM^n\otimes\mV$ is $N=2$, 
\emph{i.e.} $\RM^n\otimes\mV = \mW_1\oplus\mW_2$, in the followng two cases:
\begin{itemize}
\item[(i)] the dimension $n$ is odd and $\rho =(\tfrac12,\dots,\tfrac12)$,
\item[(ii)] the dimension $n$ is even and the dominant weight of the representation $\mV$ is 
$\rho = (k,\dots,k,\pm k)$ with $k$ an arbitrary nonzero integer or half-integer.
\end{itemize}
Case (i) is nothing but the spin representation in odd dimensions and one has $w_1=\tfrac12$, $w_2=-\tfrac{n-1}{2}$, as already noticed in the previous section. The cases covered by (ii) include the half-spin representations (when $k=\tfrac12$) as well as a lot of other representations, and the values of the weights are $w_1=k$ and $w_2=1-\tfrac{n}{2}-k$. In both cases (i) and (ii), one has $w_1>0>w_2$. Since $N=2$, there is only one non trivial Weitzenb\"ock formula (up to constant multiples), which is the universal one:
$$ w_1\,(P_1)^* P_1 + w_2\,(P_2)^*P_2 = \mathcal{R} .$$
The curvature contribution has been computed for instance in \cite{pg-weyl}:
$$ \mathcal{R} = \mathcal{R}^{\rho} = \sum_{i<j} \rho(e_i\wedge e_j)\circ \rho\left(R^g (e_i\wedge e_j)\right) $$
where $R^g$ is the curvature operator of the metric $g$.
In case (i), $(\mV,\rho)$ is the spin representation $(\Sigma,\varsigma)$, $P_1$ is a multiple of the twistor operator, and $P_2$ is a multiple of the Dirac operator; in case (ii) and if $k=\tfrac12$, $(\mV,\rho)$ is one of the half-spin representations $(\Sigma,\varsigma_\pm)$, and $P_1$, resp. $P_2$, is again a multiple of the twistor, resp. Dirac, operator (these computations have been done in section~4). From Theorem \ref{prop:universal}, one has in general
$$ \sum_{\kappa=1}^{\dim\mV} \int_M w_1 |P_1\sigma^{\kappa}|^2 + w_2 |P_2\sigma^{\kappa}|^2 - \langle \sigma^{\kappa}, \mathcal{R}^{\rho} \sigma^{\kappa}\rangle = \frac{c(\rho)}{n(n-1)}(\dim\mV)\, m(g). $$
This can be used in two different ways: 
\begin{itemize}
\item[(a)] either $\mathcal{R}^{\rho}$ is nonnegative and there exists enough asymptotically constant elements in $\ker P_1$, and we get a \emph{negative} mass theorem,
\item[(b)] or $\mathcal{R}^{\rho}$ is nonpositive and there exists enough asymptotically constant elements in $\ker P_2$, and we get a \emph{positive} mass theorem.
\end{itemize}
This is of course consistent with the classical spinorial proof of the positive mass theorem \cite{witten}: the reader
has to keep in mind that $\mathcal{R}^{\varsigma}=-\tfrac18\Scal$ when the Lichnerowicz formula for spinors is written as the universal Weitzenb\"ock formula. 
In this spinorial proof, the form $\mathcal{Q}$ is moreover \emph{diagonal}. Some further look reveals that
this is a consequence of the fact that the representation $\varsigma$ is induced by a representation of algebras (of the Clifford algebra). Thus, this is a special feature of the spinor bundle which cannot be hoped for in the general case.

\subsection{The case of forms} 
As a further example, one may 
look at the case of $p$-forms with $p<\tfrac{n}{2}$. One has in this case $\rho = (1,...,1,0,...,0)$ where the number 
of $1$'s is equal to $p$, and the number of summands is equal to $N=3$. The case where $p=\tfrac{n}{2}$ in even dimensions, 
\emph{i.e.} $\rho = (1,...,1,\pm 1)$ is special, as it has $N=2$; it has thus been considered in the previous section. 

One may then apply the previous Proposition: the only possible Weitzenb\"ock formula is the universal one and 
the conformal weights are $w_1=1$, $w_2=-p$, and $w_3=-(n-p)$. 
There are again two possible operators: either $P_+=P_1$ (and, up to a constant, this is the conformal Killing operator), or 
$P_+=\sqrt{-w_2}P_2 + \sqrt{-w_3}P_2$ (and, up to a constant again, this is
the classical Hodge-de Rham operator $d+\delta$). 
The curvature term $\mathcal{R}$ relative to $p$-forms is well-known, see for instance \cite{homma-casimir,labbi}.
The two possible boundary 
terms may also be computed from Theorem~\ref{prop:universal} and this leads to a positive mass theorem
for the latter and a negative one for the former (under the condition that $\mathcal{R}_p$ is nonpositive in the first case and 
nonnegative in the second).

\subsection{Conformally flat asymptotically flat manifolds}
We shall present here a case where the discussion of any of the last two sections effectively leads to positive or negative mass theorems. We consider the case where
$M=\RM^n$ and $g$ is globally conformal to the Euclidean metric $g_0$. We may thus write
$$ g = u^{-2}\, g_{0} $$
for some positive function $u$, such that $g$ is asymptotically flat in the sense of \ref{def_aflat}, \emph{i.e.}
$$ u - 1 = O(r^{-\tau}), \quad \partial_k u = O(r^{-\tau-1}), 
\quad \partial_k\partial_{\ell}u = O(r^{-\tau-2}) $$
for some $\tau>\tfrac{n-2}{2}$.  

Given any representation $(\mV,\rho)$, it is known that the Stein-Weiss operators $P_j=\Pi_j\circ\nabla$ (corresponding to a single projection) are conformally covariant with respect to the corresponding conformal weight \cite{fegan,pg-weyl}. This implies that, for any section $\bar{\sigma}$ of the bundle $E$,
$$ P^{0}_j(\bar{\sigma}) = 0 \ \Rightarrow \ P_j(u^{-w}\bar{\sigma}) = 0$$
with $w=w(\lambda_j, \rho)$. As constant sections $(\sigma_0^{\kappa})_{1\leq\kappa\leq\dim\mV}$ of $E$ over the Euclidean space are elements of the kernel of each $P^{0}_j$, the remark above yields a family $(\sigma^{\kappa}=u^{-w}\sigma_0^{\kappa})_{1\leq \kappa\leq\dim\mV}$ of elements of the kernel of each $P_j$, with $\sigma^{\kappa} \to_{\infty} \sigma_0^{\kappa}$.

Thus, one may obtain positivity statements on the mass if there exists a Weitzenb\"ock formula where either $P_+$ or $P_-$ is obtained from a single projection. This is for instance the case when $N=2$ or for forms, as we have seen in the last sections, and one obtains the following theorems, keeping the same notation as in the previous sections.

\begin{theo}
Let $g = u^{-2}\, g_{0}$ be an asymptotically flat and conformally flat metric on $\RM^n$ such that its mass is defined, and let $(\mV,\rho)$ be any representation of the special orthogonal group $\Spon(n)$ or its universal covering $\Spin(n)$ such that $N=2$. Then the following hold:
\begin{itemize}
\item[(i)] if $\mathcal{R}^{\rho}$ is nonpositive, then the mass $m(g)$ is nonnegative;
\item[(ii)] if $\mathcal{R}^{\rho}$ is nonnegative, then the mass $m(g)$ is nonpositive.
\end{itemize}
\end{theo}

In the case of forms, one may only obtain negative mass theorems as our analysis yields elements in the kernel of operators associated to a single projection only (Stein-Weiss operators). When the metric is
conformally flat, the curvature term is a positive multiple of 
$$\mathcal{R}_p = - \left( \frac{n-2p}{(p-1)!} \, g^{p-1}Z^g + \frac{2(n-p)}{(p-1)!}\,\Scal^g\,\textrm{Id}\right) $$
where $Z^g$ is the tracefree Ricci tensor and $g^{p-1}Z^g$ stands for its suspension as a symmetric operator acting on $p$-forms, see \cite{labbi}.

\begin{theo}
Let $g = u^{-2}\, g_{0}$ be an asymptotically flat and conformally flat metric on $\RM^n$ such that its mass is defined, and $1\leq p < \tfrac{n}{2}$. If $\mathcal{R}_p$ is nonnegative, then the mass $m(g)$ is nonpositive.
\end{theo}

\bigskip

\nocite{*}
\bibliographystyle{cdraifplain}

\end{document}